\RequirePackage{fix-cm}
\documentclass[smallextended]{svjour3}       
\smartqed  
\usepackage{csquotes}
\usepackage{graphicx}
\usepackage{amsfonts}
\usepackage{amsmath}
\usepackage{geometry}
\usepackage{color}
\usepackage[hidelinks]{hyperref}
\usepackage{cite}

\begin{document}

\title{Constrained mock-Chebyshev least squares approximation for Hermite interpolation
}


\author{Francesco Dell'Accio \and 
        Francisco Marcellán \and 
        Federico Nudo 
}


\institute{  Francesco Dell'Accio \at
             Department of Mathematics and Computer Science, University of Calabria, Rende (CS), Italy\\ 
             Istituto per le Applicazioni del Calcolo \enquote{Mauro Picone}, Naples Branch, C.N.R. National Research Council of Italy, Napoli, Italy
             \email{francesco.dellaccio@unical.it} 
         \and 
             Francisco Marcellán \at
              Departamento de Matemáticas, Universidad Carlos III de Madrid, Spain \\
              \email{pacomarc@ing.uc3m.es}
         \and 
            Federico Nudo (corresponding author) \at
              Department of Mathematics \enquote{Tullio Levi-Civita}, University of Padova, Italy \\
              \email{federico.nudo@unipd.it}
}

\date{Received: date / Accepted: date}

\maketitle

\begin{abstract}
This paper addresses the challenge of function approximation using Hermite interpolation on equally spaced nodes. In this setting, standard polynomial interpolation suffers from the Runge phenomenon. To mitigate this issue, we propose an extension of the constrained mock-Chebyshev least squares approximation technique to Hermite interpolation. This approach leverages both function and derivative evaluations, resulting in more accurate approximations. Numerical experiments are implemented in order to illustrate the effectiveness of the proposed method.
\end{abstract}

\section{Introduction}
\label{intro}
In various fields of applied mathematics, a common challenge is to approximate an unknown function $f$ over an interval  $[a,b]$, given only the function evaluations at a finite set of points 
\begin{equation}\label{pointset}
    X_n = \left\{x_0, \dots, x_n\right\}, \qquad a=x_0<x_1<\cdots<x_n=b,
\end{equation}
using the polynomial interpolation on $X_n$, that is
\begin{equation*}
 p_n\left(f,\cdot\right)\in\mathbb{P}_n, \qquad \text{s.t.} \qquad   p_n\left(f,x_i\right)=f\left(x_i\right), \qquad i=0,\dots,n,
\end{equation*}
where $\mathbb{P}_n$ is the space of polynomials of degree at most $n$ with real coefficients. For simplicity, and without loss of generality, we can assume the interval to be $[-1,1]$. A particularly interesting case arises when $X_n$ is the set of $n+1$ equispaced nodes in $[-1,1]$, that is
\begin{equation*}
    x_i=-1+\frac{2}{n}i, \qquad i=0,\dots,n.
\end{equation*}
In such instances, Lagrange polynomial interpolation can lead to catastrophic results due to the Runge phenomenon, leading to significant errors at the endpoints of the interval. This phenomenon arises from the inherent limitations of polynomial interpolation when dealing with oscillatory or rapidly changing functions. To address this issue, several techniques have been proposed, see e.g.~\cite{Boyd:2009:DRP,DeMarchi:2015:OTC,DellAccio:2022:GOT}. A notable approach, proposed in~\cite{Boyd:2009:DRP}, involves approximating $f$ with a polynomial $p_m\left(f,\cdot\right)\in\mathbb{P}_m$ of degree 
\begin{equation}\label{valm}
    m=\left\lfloor \pi \sqrt{\frac{n}{2}} \right\rfloor,
\end{equation}
obtained by interpolating $f$ not on all nodes of $X_n$, but only on a proper subset of $m+1$ points
\begin{equation}\label{mockChebnodes}
X_m^{\prime}=\left\{x_0^{\prime},\dots,x_m^{\prime}\right\}\subset X_n.    
\end{equation}
 These nodes are chosen to closely mimic the behavior of Chebyshev-Lobatto nodes~\cite{Gautschi:1997:NA} of order $m+1$, that is
\begin{equation*}
    X^{\mathrm{CL}}_m=\left\{x_0^{\mathrm{CL}},\dots,x_m^{\mathrm{CL}}\right\}, \qquad x_j^{\mathrm{CL}}=-\cos\left(\frac{\pi}{m}j\right), \qquad j=0,\dots,m.
\end{equation*}
In other words, for any $i=0,\dots,m$, the node $x_i^{\prime}\in X_m^{\prime}\subset X_n$ is defined as the solution of the following minimization problem 
\begin{equation}\label{minprob}
    \min_{k=0,\dots,n} \left\lvert x_k-x_i^{\mathrm{CL}}\right\rvert.
\end{equation}
These nodes are commonly known as \textit{mock-Chebyshev nodes}~\cite{Ibrahimoglu:2020:AFA,Ibrahimoglu:2024:ANF} and this interpolation technique is referred to as \textit{mock-Chebyshev subset interpolation}~\cite{Boyd:2009:DRP}. More specifically, the value of $m$, defined in~\eqref{valm}, is computed as the greatest value so that the set $X^{\prime}_m$ consists of $m+1$ distinct nodes~\cite{Boyd:2009:DRP}. Equivalently, $m$ is chosen as the greatest value so that the minimization problems~\eqref{minprob} each have different solutions. Using this technique, many data points are not used. Specifically, all nodes of the set 
\begin{equation*}
    X_n\setminus X_m^{\prime}
\end{equation*}
are not used. To address this limitation, the \textit{constrained mock-Chebyshev least squares approximation} method was developed in~\cite{DeMarchi:2015:OTC}. This method involves the approximation of the function $f$ with a polynomial $\hat{p}_r(f,\cdot)\in\mathbb{P}_r$ of degree 
\begin{equation}\label{valr}
    r=m+p+1, \qquad m=\left\lfloor \pi \sqrt{\frac{n}{2}} \right\rfloor, \qquad p=\left\lfloor \pi \sqrt{\frac{n}{12}} \right\rfloor,
\end{equation}
obtained by interpolating $f$ at the set of mock-Chebyshev nodes $X^{\prime}_m$ and leveraging the remaining nodes to enhance the approximation accuracy through a simultaneous regression. It has been shown that the degree~\eqref{valr} produces a good interpolation accuracy in the uniform norm, see~\cite{DeMarchi:2015:OTC}. 
The constrained mock-Chebyshev least squares approximation, as introduced in~\cite{DeMarchi:2015:OTC}, is defined through the nodal polynomial on the set 
$X_m^{\prime}$.  This approach, although elegant, suffers from limitations: it cannot be generalized to the bivariate case and is not suitable for many applications.

In~\cite{DellAccio:2022:GOT}, the constrained mock-Chebyshev least squares approximation was successfully extended to the multivariate case, introducing a new computational approach that broadened its applicability, see~\cite{DellAccio:2022:CMC,DellAccio:2022:AAA,DellAccio:2023:PIR,DellAccio:2024:PAO,DellAccio:2024:NAO}. Very recently, this approximation operator was further extended to leverage the zeros of general orthogonal polynomials due to their beneficial properties~\cite{DellAccio:2024:AEO}.

However, in many applications, it is often feasible to have access not only to the functional evaluations at the points of $X_n$, but also to the evaluations of its first few derivatives. In such a scenario, a common approach is to employ Hermite interpolation. Hermite interpolation is a form of polynomial interpolation where both the function values and the values of its derivatives at given points are known. This method allows to construct a polynomial that not only passes through the specified points but also matches the specified derivatives at those points.
This form of interpolation is particularly useful in scenarios where the function and its rate of change are known at certain points, ensuring a more accurate approximation of the original function. Assuming to know, for each node $x_i\in X_n$, the data
\begin{equation*}
    f^{(0)}\left(x_i\right), f^{(1)}\left(x_i\right),\dots, f^{(k)}\left(x_i\right), \qquad k\ge 1,
\end{equation*}
where $f^{(0)}=f$, Hermite interpolation consists to find a polynomial of degree 
\begin{equation}\label{widen}
\widetilde{n}=\left(k+1\right)\left(n+1\right)-1
\end{equation}  
denoted as $H_{k,n}\left(f,\cdot\right)$ such that
\begin{equation}\label{impperthm}
H^{(\ell)}_{k,n}\left(f,x_i\right)=f^{(\ell)}\left(x_i\right), \qquad i=0,\dots,n, \qquad \ell=0,\dots,k.
\end{equation}
Hermite interpolation provides a better approximation than simple polynomial interpolation by ensuring a faithful representation of both function values and their derivatives, leading to more precise interpolations. This makes it a valuable tool in various scientific and engineering applications, such as spline fitting for complex shapes or solving differential equations where capturing the rate of change is critical~\cite{Neuman:1978:UAB,Meek:1997:GHI, Bertolazzi:2018:OTG, Fageot:2020:SAA}. Moreover, this type of interpolation is widely used in numerical analysis, computer graphics, and data fitting~\cite{DeBoor:1987:HAG,Farouki:1995:HIBhermite, Juttler:2001:HIB}. 
It is worth noting that Hermite interpolation using the zeros of Chebyshev polynomials of the first kind can be used to prove the Weierstrass theorem in Fejér's proof, as detailed in~\cite[Th. 6.4.1]{Davis:1975:IAA}. This highlights an interesting application of such interpolation in the uniform approximation theory of continuous functions by polynomials.
For this type of approximation, the following theorems hold~\cite{Davis:1975:IAA}. Throughout the paper, we assume that $k\in\mathbb{N}$ is a fixed integer such that $k\ge 1$.
\begin{theorem}\label{th1}
   Let $f\in C^k[a,b]$ and let $X_n=\left\{x_0,\dots,x_n\right\}$ be the point set defined in~\eqref{pointset}. Let $\widetilde{n}$ be the value defined in~\eqref{widen}. 
   Then, there exists a unique polynomial of degree $\widetilde{n}$
   \begin{equation*}
       H_{k,n}\left(f,\cdot\right)\in \mathbb{P}_{\widetilde{n}}, 
   \end{equation*}
   which satisfies~\eqref{impperthm}.
\end{theorem}

\begin{theorem}
Let $f\in C^k[a,b]$ and let $X_n=\left\{x_0,\dots,x_n\right\}$ be the point set as in~\eqref{pointset}. Let $\widetilde{n}$ be the value defined in~\eqref{widen}. We assume that $f^{(\widetilde{n}+1)}(x)$ exists for any $x\in(a, b)$. Denote by $H_{k,n}\left(f,\cdot\right)\in \mathbb{P}_{\widetilde{n}}$ the unique polynomial which satisfies~\eqref{impperthm}.
 Then 
\begin{equation}\label{errorHermite}
        f(x)-H_{k,n}\left(f,x\right)=\frac{f^{\left(\widetilde{n}+1\right)}\left(\xi(x)\right)}{\left(\widetilde{n}+1\right)!}\left[\left(x-x_0\right)\cdots\left(x-x_n\right)\right]^k, \qquad x\in(a,b),
\end{equation}
where 
\begin{equation*}
    \min\left\{x,x_0\right\}<\xi(x)<\max\left\{x,x_n\right\}.
\end{equation*}
\end{theorem}
The previous theorem relates the position of the nodes in the set $X_n$ to the interpolation error. More precisely, the error formula includes the term 
\begin{equation}\label{polnodk}
    \left[\left(x-x_0\right)\cdots\left(x-x_n\right)\right]^k,
\end{equation}
which is the nodal polynomial relative to the interpolation nodes of $X_n$ up to the power of $k$. It is well known that the configuration of points that minimizes the polynomial~\eqref{polnodk} is the set of Chebyshev nodes~\cite{Gautschi:1997:NA}. However, when the function is known only at the equispaced points, the evaluations of $f$ on the Chebyshev nodes are not available. Then, in this case, the constrained mock-Chebyshev least squares approximation method can be used to reduce the interpolation error. 

The main goal of this paper is to extend the constrained mock-Chebyshev least squares approximation technique to the case of Hermite interpolation. This will be the main topic of Section~\ref{sec1}. The accuracy of the proposed method is pointed out through a series of numerical experiments presented in Section~\ref{SecNum}.

\section{Constrained mock-Chebyshev least squares approximation for Hermite interpolation}
\label{sec1}
Let $n\in\mathbb{N}$ and let 
\begin{equation*}
    X_n=\left\{x_0,\dots,x_n\right\}, \qquad x_i=-1+\frac{2}{n}i, \qquad i=0,\dots,n,
\end{equation*}
be the set of $n+1$ equispaced nodes in $[-1,1]$. We denote by  
\begin{equation*}
    X_m^{\prime}=\left\{x_0^{\prime},\dots,x_m^{\prime}\right\}\subset X_n, \qquad m=\left\lfloor \pi \sqrt{\frac{{n}}{2}}\right\rfloor,
\end{equation*}
the subset of mock-Chebyshev nodes defined in~\eqref{mockChebnodes}.
Let $f\in C^k[-1,1]$, $k\ge1$, be a smooth function defined in $[-1,1]$ and we assume to know the values
\begin{equation*}
f^{(0)}\left(x_i\right),f^{(1)}\left(x_i\right),\dots,f^{(k)}\left(x_i\right), \qquad i=0,\dots,n,
\end{equation*}
where $f^{(0)}=f$.
In order to extend the constrained mock-Chebyshev least squares approximation to Hermite interpolation, specific settings are necessary. Unlike the standard constrained mock-Chebyshev least squares interpolation, where the available data consist of $n+1$ function evaluations, in Hermite interpolation we have $\widetilde{n}+1=(k+1)(n+1)$ data, consisting of function evaluations and derivative evaluations up to order $k$. Drawing an analogy to standard constrained mock-Chebyshev least squares approximation, we set
\begin{equation}\label{valrtilde}
    \widetilde{r} = (k+1)r = (k+1)(m+p+1),
\end{equation}
where $r$ is defined in~\eqref{valr}.

We consider a basis
\begin{equation*}
    \mathcal{B}_{\widetilde{r}}=\left\{u_0,\dots,u_{\widetilde{r}}\right\}
\end{equation*}
of the polynomial space $\mathbb{P}_{\widetilde{r}}$ such that
\begin{equation*}
    \operatorname{span}\left\{u_0,\dots,u_{{m}^{\star}}\right\}=\mathbb{P}_{m^{\star}},
\end{equation*}
where 
\begin{equation}\label{mstar}
    m^{\star}=(k+1)(m+1)-1.
\end{equation} 
For simplicity, in the sequel we assume that the set $X_n$ has been reordered so that its first $m+1$ elements are those of the set $X_m^{\prime}$, that is,
\begin{equation*}
    x_i=x_i^{\prime}, \qquad i=0,\dots,m, 
\end{equation*}
and after that the other nodes.
After these rearrangements, we consider the following vectors 
\begin{eqnarray}\label{nuov1}
    \boldsymbol{v}^{(i)}&=&\left[f^{(i)}\left(x_0\right),\dots,f^{(i)}\left(x_m\right)\right]\in\mathbb{R}^{m+1}, \qquad i=0,\dots,k,\\
    \label{nuov2}\boldsymbol{w}^{(i)}&=&\left[f^{(i)}\left(x_{m+1}\right),\dots,f^{(i)}\left(x_n\right)\right]\in\mathbb{R}^{n-m},\qquad i=0,\dots,k,
\end{eqnarray}
and matrices
\begin{equation*}
    A^{(i)}=\begin{bmatrix}
u^{(i)}_0\left(x_0\right) & u^{(i)}_1\left(x_0\right) & \cdots & u^{(i)}_{\widetilde{r}}\left(x_0\right)\\
u^{(i)}_0\left(x_1\right) & u^{(i)}_1\left(x_1\right) & \cdots & u^{(i)}_{\widetilde{r}}\left(x_1\right)\\
\vdots  & \vdots  & \ddots & \vdots  \\
u^{(i)}_0\left(x_m\right) & u^{(i)}_1\left(x_m\right) & \cdots & u^{(i)}_{\widetilde{r}}\left(x_m\right)\\
\end{bmatrix}\in \mathbb{R}^{(m+1)\times(\widetilde{r}+1)}, \qquad i=0,\dots,k,\\
\end{equation*}
\begin{equation*}
    B^{(i)}=\begin{bmatrix}
u^{(i)}_0\left(x_{m+1}\right) & u^{(i)}_1\left(x_{m+1}\right) & \cdots & u^{(i)}_{\widetilde{r}}\left(x_{m+1}\right)\\
u^{(i)}_0\left(x_{m+2}\right) & u^{(i)}_1\left(x_{m+2}\right) & \cdots & u^{(i)}_{\widetilde{r}}\left(x_{m+2}\right)\\
\vdots  & \vdots  & \ddots & \vdots  \\
u^{(i)}_0\left(x_{n}\right) & u^{(i)}_1\left(x_{n}\right) & \cdots & u^{(i)}_{\widetilde{r}}\left(x_n\right)\\
\end{bmatrix}\in \mathbb{R}^{(n-m)\times(\widetilde{r}+1)}, \qquad i=0,\dots,k.
\end{equation*}
By using this notation, we define the following column vectors 
\begin{eqnarray}    \boldsymbol{b}=\boldsymbol{b}(f)&=&\left[\boldsymbol{v}^{(0)},\dots, \boldsymbol{v}^{(k)},\boldsymbol{w}^{(0)},\dots,\boldsymbol{w}^{(k)}\right]^T\in\mathbb{R}^{\widetilde{n}+1}, \label{bvett}\\ \boldsymbol{d}=\boldsymbol{d}(f)&=&\left[\boldsymbol{v}^{(0)},\dots, \boldsymbol{v}^{(k)}\right]^T\in\mathbb{R}^{m^{\star}+1},
\end{eqnarray}
and the block-defined matrices
\begin{equation*}
    \Lambda=\begin{bmatrix}
A^{(0)}\\
A^{(1)}\\
\vdots  \\
A^{(k)}\\
B^{(0)}\\
B^{(1)}\\
\vdots  \\
B^{(k)}\\
\end{bmatrix}\in \mathbb{R}^{(\widetilde{n}+1)\times(\widetilde{r}+1)}, \qquad \Xi=\begin{bmatrix}
A^{(0)}\\
A^{(1)}\\
\vdots  \\
A^{(k)}\\
\end{bmatrix}\in \mathbb{R}^{(m^{\star}+1)\times(\widetilde{r}+1)},
\end{equation*}
where $\widetilde{n}$ and $m^{\star}$ are defined in~\eqref{widen} and~\eqref{mstar}, respectively. 
Finally, we define the approximation operator
\begin{equation}\label{operatorPhat}
\begin{array}{rcl}
\hat{H}_{\widetilde{r},n,k}:  f\in C^k[-1,1] &\mapsto& \hat{H}_{\widetilde{r},n,k}\left(f,x\right)=\displaystyle\sum\limits_{j=0}^{\widetilde{r}} c_j(f)u_j(x)\in \mathbb{P}_{\widetilde{r}},
\end{array}
\end{equation}
where the vector of coefficients ${\boldsymbol{c}}(f)=\left[{c}_0(f),{c}_1(f),\dots,{c}_{\widetilde{r}}(f)\right]^T$ satisfies
\begin{equation}\label{KKTlinearsystem}
    \begin{bmatrix}
\Delta &  \Xi^T  \\
 \Xi & 0  \\
\end{bmatrix}
\begin{bmatrix}
{\boldsymbol{c}}(f) \\
{\boldsymbol{z}} \\
\end{bmatrix}=
\begin{bmatrix}
2\Lambda^T\boldsymbol{b}\\
\boldsymbol{d} \\
\end{bmatrix}. 
\end{equation}
Here
\begin{equation*}
    \Delta=2 \Lambda^T\Lambda
\end{equation*}
and ${\boldsymbol{z}}$ is the vector of Lagrange multipliers. The matrix of the linear system~\eqref{KKTlinearsystem} 
\begin{equation}\label{KKTmatrix}
    K= \begin{bmatrix}
\Delta &  \Xi^T  \\
 \Xi & 0  \\
\end{bmatrix}
\end{equation}
is called KKT matrix in honor
to W. Karush, H. Kuhn and A. Tucker. 
\begin{remark}
    We observe that the approximation operator $\hat{H}_{\widetilde{r},n,k}$ satisfies
    \begin{equation*}
        \hat{H}_{\widetilde{r},n,k}^{(\ell)}\left(f,x_i\right) = f^{(\ell)}\left(x_i\right), \quad i = 0,\dots,m, \quad \ell = 0,\dots,k,
    \end{equation*}
    for any $f \in C^k[-1,1]$. This approximation generally provides a more accurate result compared to the standard constrained mock-Chebyshev least squares operator because it leverages additional information beyond functional evaluations at specified points. This is particularly useful for capturing the rate of change of $f$.
\end{remark}

\begin{remark}
    We notice that, for $k=0$, the degree $\widetilde{r}$ coincides with the degree $r$ used in the constrained mock-Chebyshev least squares approximation. Then, in this sense, the approximation operator~\eqref{operatorPhat} generalizes the constrained mock-Chebyshev least squares approximation.
\end{remark}

The following theorem shows that the approximation operator~\eqref{operatorPhat} is well-defined.
\begin{theorem}
\label{unicity}
The KKT matrix~\eqref{KKTmatrix} is nonsingular.
\end{theorem}
\begin{proof}
   To demonstrate this theorem, it suffices to prove that the matrices $\Lambda$ and $\Xi$ have maximum rank~\cite[Ch. 16]{Boyd:2018:ITA}.
   By Theorem~\ref{th1}, there exists a unique  polynomial 
   \begin{equation*}
       H_{k,n}\left(f,\cdot\right)\in \mathbb{P}_{\widetilde{n}}, 
   \end{equation*}
such that
\begin{equation}
H^{(\ell)}_{k,n}\left(f,x_i\right)=f^{(\ell)}\left(x_i\right), \qquad i=0,\dots,n,\qquad \ell=0,\dots,k.
\end{equation}
Then, the relative Vandermonde matrix $\Lambda^{\mathrm{ext}}$ with respect to the same data and to an extension of the basis $\mathcal{B}_{\widetilde{r}}$ to the space $\mathbb{P}_{\widetilde{n}}$
   \begin{equation}\label{bases}
       B_{\widetilde{n}}=\left\{u_0,\dots,u_{\widetilde{r}}, u_{\widetilde{r}+1}, \dots, u_{\widetilde{n}}\right\}
   \end{equation}
is a nonsingular matrix. We notice that the size of $\Lambda^{\mathrm{ext}}$ is $\left(\widetilde{n}+1\right)\times\left(\widetilde{n}+1\right)$ and that $\Lambda$ is its submatrix formed by its first $\widetilde{r}+1$ columns. Thus, since $\widetilde{r}<\widetilde{n}$, the matrix $\Lambda$ has maximum rank.

It remains to prove that the matrix $\Xi$ has maximum rank. To this aim, it is sufficient to observe that the submatrix $\Xi^{\mathrm{sub}}$ of $\Xi$ formed by its first $m^{\star}+1$ columns is a nonsingular matrix, since it is the Vandermonde matrix relative to the data 
\begin{equation*}
\left\{f^{(\ell)}\left(x_i\right)\,:\, i=0,\dots,m,\quad \ell=0,\dots,k,\right\},
\end{equation*}
with respect to the basis 
   \begin{equation}\label{basesm}
       B_{m^{\star}}=\left\{u_0,\dots,u_{m^{\star}}\right\}.
   \end{equation} 
Thus, since $\widetilde{r}>m^{\star}$, the matrix $\Xi$ has maximum rank.
\end{proof}

\begin{remark}
    A direct consequence of Theorem~\ref{unicity} is that, for any $f\in C^k[-1,1]$, the coefficients ${c}_j(f)$ of the polynomial $\hat{H}_{\widetilde{r},n,k}\left(f,\cdot\right)$ are uniquely determined. Consequently, the approximation operator~\eqref{operatorPhat} is well-defined.
\end{remark}

\begin{remark}
We note that the condition number of the KKT matrix $K$, defined in~\eqref{KKTmatrix}, increases as the number of derivatives involved in the problem rises. Specifically, while the condition number of the KKT matrix for the standard constrained mock-Chebyshev least squares problem remains on the order of $10^6$ for approximately $n=10000$ data points, see~\cite{DellAccio:2024:AEO}, the condition number of the KKT matrix related to the Hermite problem with $k=1$, which involves both functional evaluations and first derivatives, escalates to $10^{17}$ for just $n=1000$, see Fig.~\ref{cond}. This corresponds to $2000$ data points, with $1000$ functional evaluations and $1000$ first derivative evaluations.
\end{remark}

\begin{figure}[ht]
\centering
\includegraphics[width=.59\textwidth]{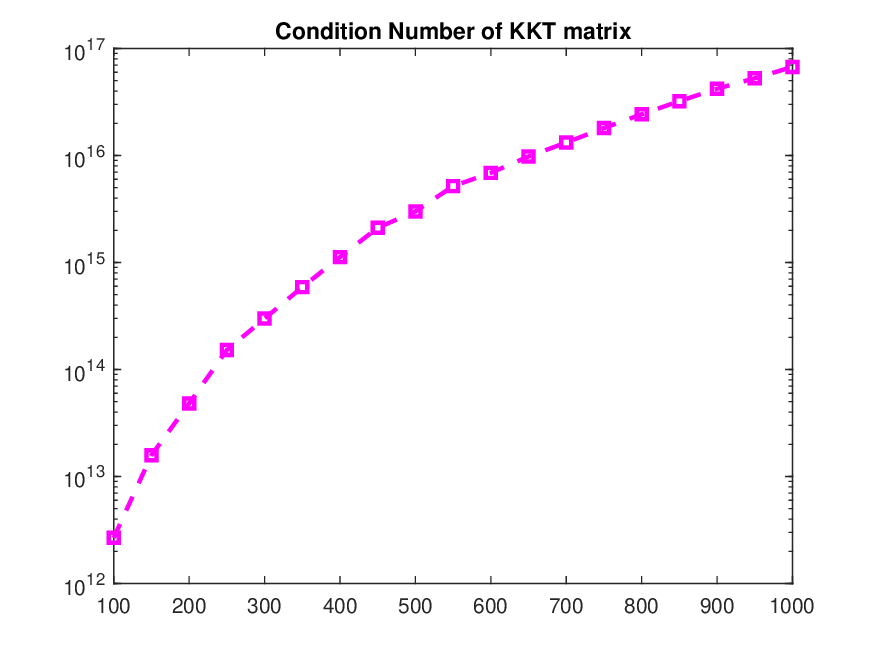}\hfil
    \caption{Trend of the condition number of the KKT matrix relative to the Hermite problem with $k=1$.
}
 \label{cond} 
\end{figure}

In the next theorem, we present some theoretical aspects of the approximation operator $\hat{H}_{\widetilde{r},n,k}$. To this end, we observe that the operator $\hat{H}_{\widetilde{r},n,k}$ depends only on the evaluations of the function $f$ and its first $k$ derivatives on a set of $n+1$ equally spaced points, for a total of $(k+1)(n+1)$ data points. For simplicity, we can define the operator~\eqref{operatorPhat} as follows:
\begin{equation}\label{discop}
\begin{array}{rcl}
\hat{H}_{\widetilde{r},n,k}: \boldsymbol{b}\in \mathbb{R}^{\widetilde{n}+1} &\mapsto& \hat{H}_{\widetilde{r},n,k}\left(\boldsymbol{b},x\right)=\displaystyle\sum\limits_{j=0}^{\widetilde{r}} c_j(\boldsymbol{b})u_j(x)\in \mathbb{P}_{\widetilde{r}},
\end{array}
\end{equation}
where $\widetilde{n}$ is defined in~\eqref{widen}, and
\begin{equation*}
\boldsymbol{b}=\begin{bmatrix}
b_0\\
\vdots\\
b_{\widetilde{n}}
\end{bmatrix}\in\mathbb{R}^{\widetilde{n}+1}.
\end{equation*}
\begin{remark} \label{remarkimp}
   We note that definitions~\eqref{operatorPhat} and~\eqref{discop} are equivalent. In fact, if we fix a function $f\in C^k[-1,1]$ and consider $\boldsymbol{b}=\boldsymbol{b}(f)$ as defined in~\eqref{bvett}, then we obtain
   \begin{equation*}
    \hat{H}_{\widetilde{r},n,k}(f,x)=\hat{H}_{\widetilde{r},n,k}(\boldsymbol{b}(f),x), \qquad x\in[-1,1].
   \end{equation*}
   Conversely, for any vector $\boldsymbol{b}\in\mathbb{R}^{\widetilde{n}+1}$, there exists a function $f$ such that  $\boldsymbol{b}=\boldsymbol{b}(f)$ as~\eqref{bvett}. For this reason, in the next theorem, we will explicitly write the components of $\boldsymbol{b}=\boldsymbol{b}(f)$ as a function of $f$. 
\end{remark}
By providing $\mathbb{R}^{\widetilde{n}+1}$ with the discrete 1-norm 
\begin{equation*}
    \left\lVert \boldsymbol{y}\right\rVert_1=\sum_{i=0}^{\widetilde{n}}\left\lvert y_i\right\rvert, \qquad \boldsymbol{y}=\begin{bmatrix}
y_0\\
\vdots\\
y_{\widetilde{n}} \\
\end{bmatrix}\in \mathbb{R}^{\widetilde{n}+1},
\end{equation*}
and $\mathbb{P}_{\widetilde{r}}$ with the $L_1$-norm
\begin{equation*}
    \left\lVert p\right\rVert_1=\int_{-1}^{1}\left\lvert p(x)\right\rvert dx, \qquad p\in \mathbb{P}_{\widetilde{r}},
\end{equation*}
 we want to find a bound for the norm 
\begin{equation}
\label{eq:NormOfOperator}
\left\lVert \hat{H}_{\widetilde{r},n,k} \right\rVert:= \sup\limits_{\substack{\boldsymbol{b}\in \mathbb{R}^{\widetilde{n}+1}\\ \left\lVert \boldsymbol{b} \right\rVert_{1}\le1}} \left\lVert \hat{H}_{\widetilde{r},n,k}(\boldsymbol{b},\cdot) \right\rVert_{1}.
\end{equation} 

\begin{theorem}
\label{Thm:EstimateOfNorm}
The norm of the operator $\hat{H}_{\widetilde{r},n,k}$ satisfies
\begin{equation}
\label{eq:BoundThm}
    \left \lVert \hat{H}_{\widetilde{r},n,k} \right\rVert\le C_1\left\lVert
 K^{-1}
\right\rVert_1\left(2 C_2\left(\widetilde{r}+1\right)+1\right),
\end{equation}
where 
\begin{equation}
    \label{eq:ConstantC}
C_1=C_1(\widetilde{r}):=\max_{j=0,\dots,\widetilde{r}}\left\lVert u_j \right\rVert_{1}, \qquad C_2=C_2(\widetilde{r},n,k):=\max_{i,j,\ell} \left\lvert u_j^{(\ell)}(x_i) \right\rvert.
\end{equation}
\end{theorem}
\begin{proof}
Let $\boldsymbol{b}\in \mathbb{R}^{\widetilde{n}+1}$ be a vector such that
\begin{equation*}
\left\lVert \boldsymbol{b} \right\rVert_{1}\le 1.    
\end{equation*}
Using~\eqref{discop} and the triangular inequality, we get
\begin{equation}
\label{eq:ContinuityCMCLSO}
    \left\lvert \hat{H}_{\widetilde{r},n,k}(\boldsymbol{b},x) \right\rvert=\left\lvert \sum\limits_{j=0}^{\widetilde{r}} c_j(\boldsymbol{b})u_j(x) \right\rvert\le \sum_{j=0}^{\widetilde{r}} \left\lvert c_j(\boldsymbol{b})\right\rvert \left\lvert u_j(x)\right\rvert, \quad x\in[-1,1].
\end{equation}
Integrating both sides of~\eqref{eq:ContinuityCMCLSO} over the interval $[-1,1]$, we obtain
\begin{equation} 
\label{eq:BoundP(f)}
   \left \lVert \hat{H}_{\widetilde{r},n,k}(\boldsymbol{b},\cdot) \right\rVert_{1} \le  \sum_{j=0}^{\widetilde{r}} \left\lvert c_j(\boldsymbol{b})\right\rvert \left\lVert u_j \right\rVert_1 \le C_1 \sum_{j=0}^{\widetilde{r}} \left\lvert c_j(\boldsymbol{b})\right\rvert=C_1\left\lVert \boldsymbol{c}(\boldsymbol{b})\right\rVert_1,
\end{equation}
where $C_1$ is defined in~\eqref{eq:ConstantC}.

It remains to find a bound for $\left\lVert \boldsymbol{c}(\boldsymbol{b})\right\rVert_1$. For this purpose, by Remark~\ref{remarkimp}, for any vector $\boldsymbol{b}\in\mathbb{R}^{\widetilde{n}+1}$, there exists a function $f\in C^k[-1,1]$ such that
\begin{equation*}
\boldsymbol{b}=\boldsymbol{b}(f)
\end{equation*}
as in equation~\eqref{bvett} and then
\begin{equation}\label{condbandf}
    \left\lVert \boldsymbol{b}\right\rVert_1=\sum_{j=0}^{\widetilde{r}} \left\lvert b_j\right \rvert= \sum_{\ell=0}^k \sum_{i=0}^n \left\lvert f^{(\ell)}(x_i) \right\rvert\le1.
\end{equation}
By Theorem~\ref{unicity} and by~\eqref{condbandf}, we have
\begin{eqnarray*}
\left\lVert
{\boldsymbol{c}(\boldsymbol{b})} 
\right\rVert_1
\le
\left\lVert
\begin{bmatrix}
\boldsymbol{c}(\boldsymbol{b}) \\
\boldsymbol{z} \\
\end{bmatrix}
\right\rVert_1&\le&
\left\lVert
 K^{-1}
\right\rVert_1
\left\lVert
\begin{bmatrix}
2 \Lambda^T\boldsymbol{b} \\
\boldsymbol{d} \\
\end{bmatrix}
\right\Vert_1\\
&=& \left\lVert
 K^{-1}
\right\rVert_1\left(2 \sum_{j=0}^{\widetilde{r}}\left\lvert\sum_{\ell=0}^k\sum_{i=0}^n u^{(\ell)}_j(x_i)f^{(\ell)}(x_i)\right\rvert +\sum_{\ell=0}^k\sum_{i=0}^m \left\lvert f^{(\ell)}(x_i)\right\rvert\right) \\
&\le& \left\lVert
 K^{-1}
\right\rVert_1\left(2C_2 \sum_{j=0}^{\widetilde{r}}\sum_{\ell=0}^k\sum_{i=0}^n   \left\lvert f^{(\ell)}(x_i)\right\rvert +\sum_{\ell=0}^k\sum_{i=0}^m \left\lvert f^{(\ell)}(x_i)\right\rvert\right) \\
&\le& \left\lVert
 K^{-1}
\right\rVert_1\left(2 C_2\left(\widetilde{r}+1\right)+1\right),
\end{eqnarray*}
where $C_2$ is defined in~\eqref{eq:ConstantC}.
Therefore, by~\eqref{eq:BoundP(f)}, we get
\begin{equation}
\label{eq:polval}
   \left \lVert \hat{H}_{\widetilde{r},n,k}(\boldsymbol{b},\cdot) \right\rVert_{1} \le C_1\left\lVert
 K^{-1}
\right\rVert_1\left(2 C_2\left(\widetilde{r}+1\right)+1\right),
\end{equation}
and then
\begin{equation*}
\left \lVert \hat{H}_{\widetilde{r},n,k} \right\rVert\le C_1\left\lVert
 K^{-1}
\right\rVert_1\left(2 C_2\left(\widetilde{r}+1\right)+1\right)
\end{equation*}
since the right-hand side of~\eqref{eq:polval} does not depend on $\boldsymbol{b}$.
\end{proof}

\section{Numerical experiments}
\label{SecNum}
In this section, we perform some numerical experiments which demonstrate the accuracy of the
proposed method. We consider the following test functions:
\begin{equation*}
    f_1(x)=\frac{1}{1+25x^2}, \quad f_2(x)=\frac{1}{1+8x^2}, \quad f_3(x)=\cos\left(50x\right), \quad f_4(x)=\frac{1}{x-1.05}.
\end{equation*}
We compute the mean approximation error for the functions $f_1$-$f_4$ and their first derivatives using the operator $\hat{H}_{\widetilde{r},n,1}$ and its first derivative $\hat{H}_{\widetilde{r},n,1}^{\prime}$ for different values of $n$ ($e_{\text{mean}}^{\mathrm{H}}$). Specifically, we compute the trend of the mean approximation error produced by the operators $\hat{H}_{\widetilde{r},n,1}$ and $\hat{H}^{\prime}_{\widetilde{r},n,1}$ from $n=100$ to $n=1000$ with stepsize $50$. These are compared against the trend on the mean approximation error produced by standard constrained mock-Chebyshev least squares approximation ($e_{\text{mean}}^{\mathrm{MC}}$) using the same number of data.  More precisely, since this approximation method relies solely on functional evaluations, we compute the approximation error using a uniform grid of $\widetilde{n}+1=2n+2$ points, where $n$ varies from $n=100$ to $n=1000$ with stepsize $50$.
The mean approximation error is computed relative to a grid of $N=132$ equally spaced nodes. The results of these experiments are shown in Figures~\ref{f1fig},\ref{f2fig},\ref{f3fig},\ref{f4fig}. We observe that, as in the case of the standard constrained mock-Chebyshev least squares approximation, the error produced by $\hat{H}_{\widetilde{r},n,1}$ and its derivative decreases as $n$ increases until a certain level of accuracy is achieved. Beyond this point, increasing $n$ does not further improve the accuracy; instead, it remains constant. Moreover, we can observe that, for several functions, the operator $\hat{H}_{\widetilde{r},n,1}$ and its derivative produce much better results compared to the standard constrained mock-Chebyshev least squares approximation. In particular, the error produced by the latter decreases not as rapidly as the error produced by $\hat{H}_{\widetilde{r},n,1}$ and its derivative.
In some cases, the constrained mock-Chebyshev least squares method yields slightly better results for very high values of $n$. This behavior is due to the increased conditioning number of the KKT matrix associated with the operator $\hat{H} _{\widetilde{r},n,1}$.

\begin{figure}[ht]
\centering
\includegraphics[width=.49\textwidth]{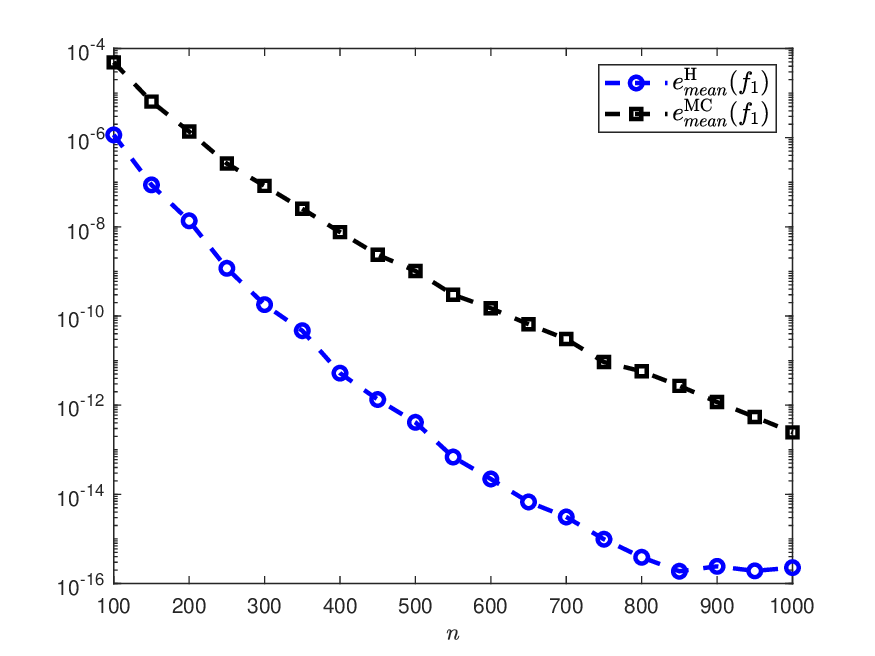}\hfil
\includegraphics[width=.49\textwidth]{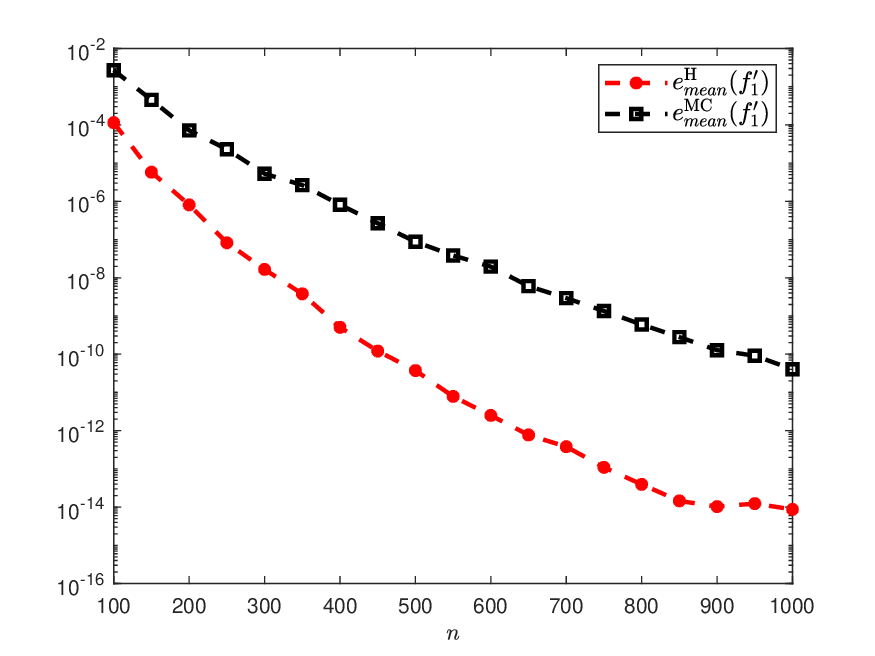}\hfil
    \caption{Mean approximation error produced by approximating the function $f_1$ (left) and $f_1^{\prime}$ (right) using the operator $\hat{H}_{\widetilde{r},n,1}$ and its first derivative $\hat{H}_{\widetilde{r},n,1}^{\prime}$, respectively. The parameter $n$ varies from $n=100$ to $n=1000$ with stepsize $50$. This is compared against the standard constrained mock-Chebyshev least squares approximation using the same number of data points.}
 \label{f1fig} 
\end{figure}

\begin{figure}[ht]
\centering
\includegraphics[width=.49\textwidth]{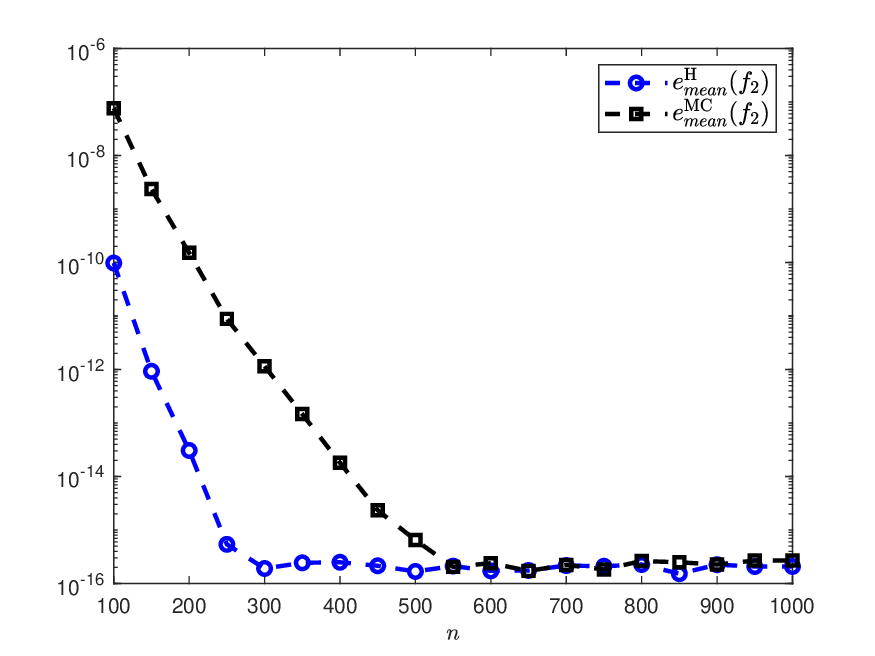}\hfil
\includegraphics[width=.49\textwidth]{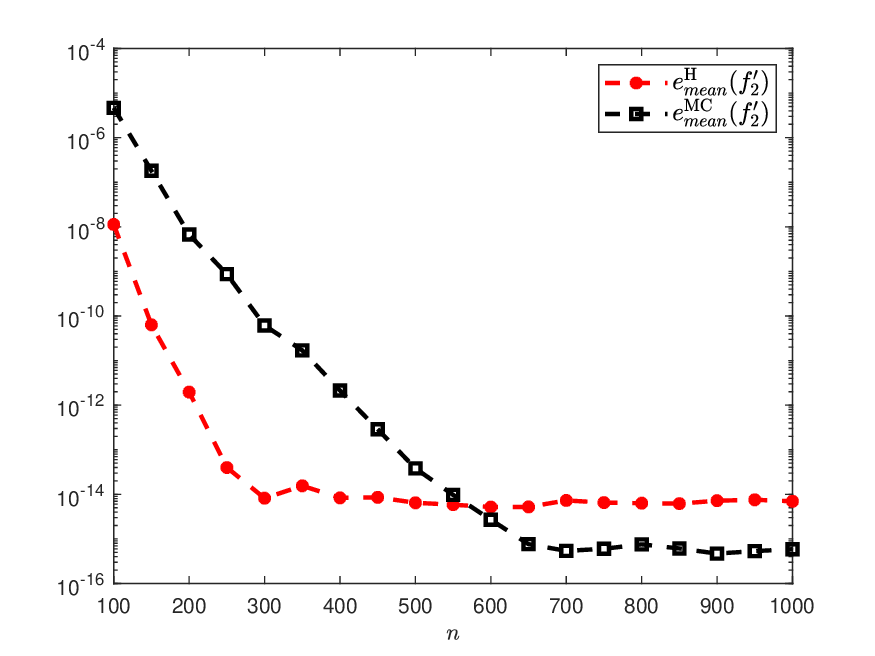}\hfil
    \caption{Mean approximation error produced by approximating the function $f_2$ (left) and $f_2^{\prime}$ (right) using the operator $\hat{H}_{\widetilde{r},n,1}$ and its first derivative $\hat{H}_{\widetilde{r},n,1}^{\prime}$, respectively. The parameter $n$ varies from $n=100$ to $n=1000$ in steps of $50$. This is compared against the standard constrained mock-Chebyshev least squares approximation using the same number of data points.
}
 \label{f2fig} 
\end{figure}

\begin{figure}[ht]
\centering
\includegraphics[width=.49\textwidth]{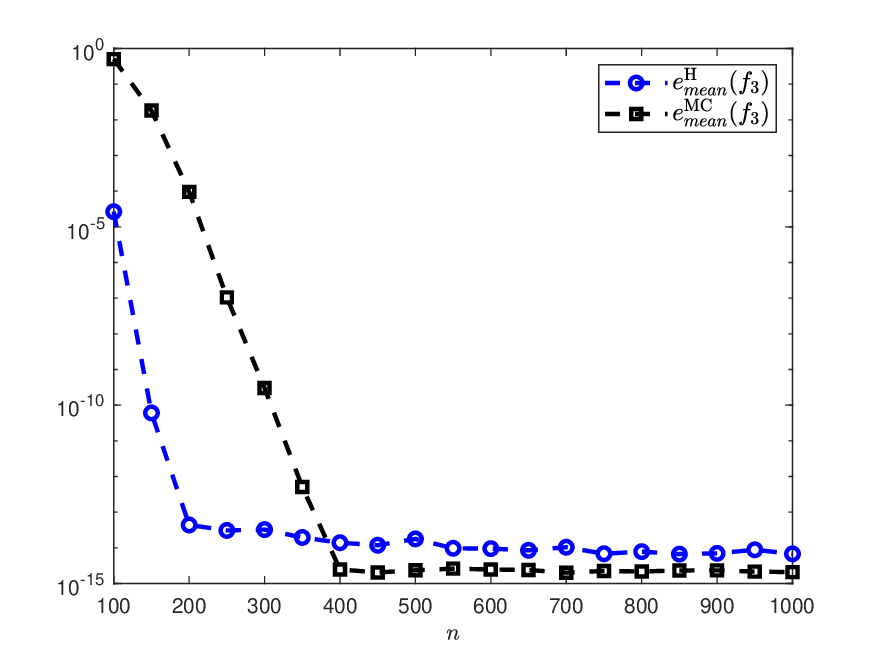}\hfil
\includegraphics[width=.49\textwidth]{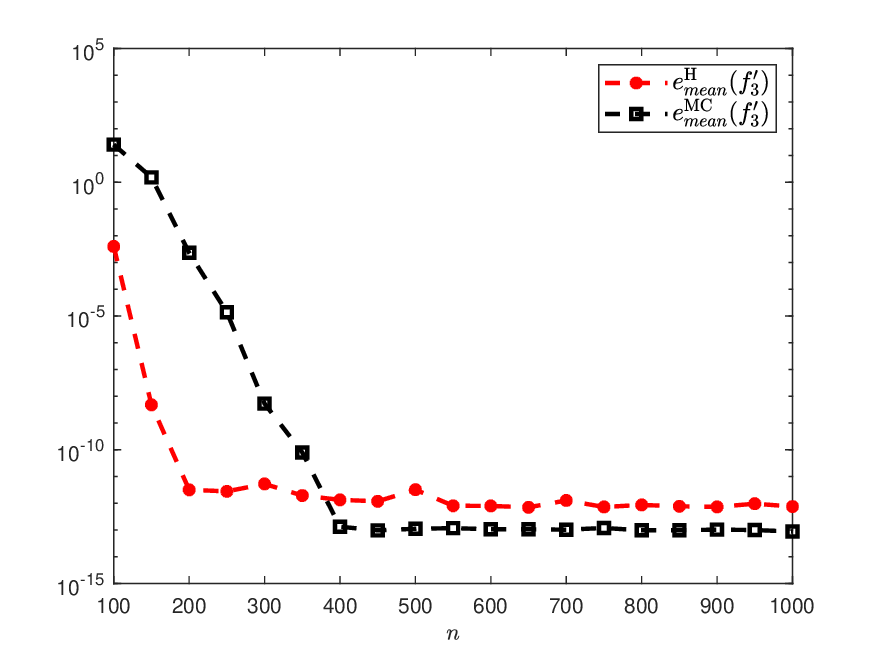}\hfil
    \caption{Mean approximation error produced by approximating the function $f_3$ (left) and $f_3^{\prime}$ (right) using the operator $\hat{H}_{\widetilde{r},n,1}$ and its first derivative $\hat{H}_{\widetilde{r},n,1}^{\prime}$, respectively. The parameter $n$ varies from $n=100$ to $n=1000$ in steps of $50$. This is compared against the standard constrained mock-Chebyshev least squares approximation using the same number of data points.}
 \label{f3fig} 
\end{figure}

\begin{figure}[ht]
\centering
\includegraphics[width=.49\textwidth]{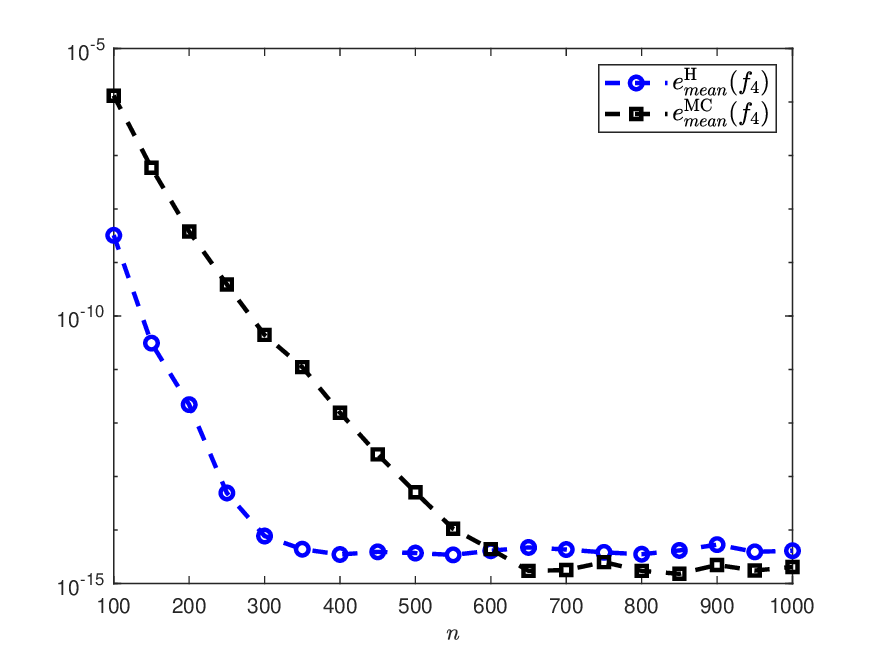}\hfil
\includegraphics[width=.49\textwidth]{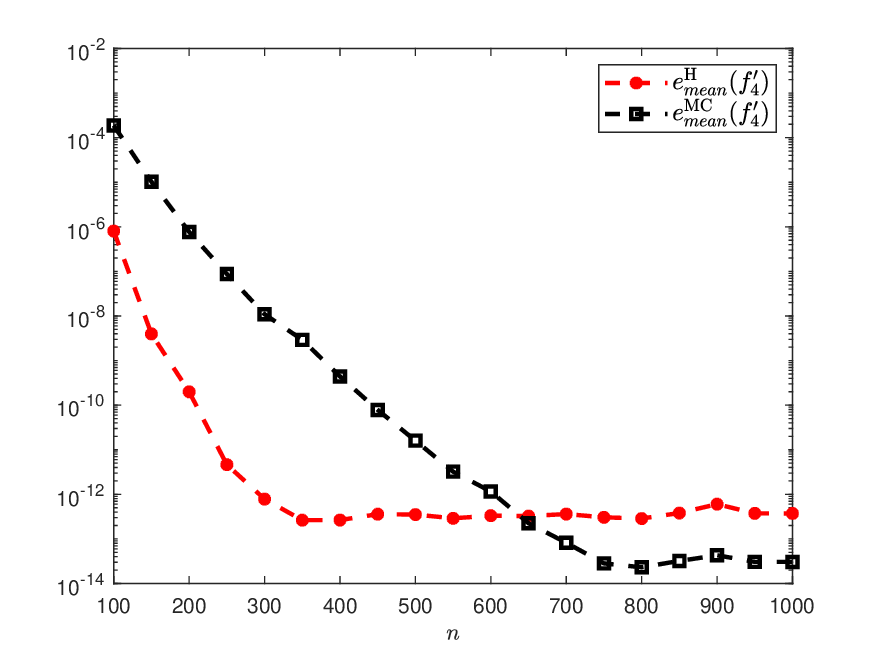}\hfil
    \caption{Mean approximation error produced by approximating the function $f_4$ (left) and $f_4^{\prime}$ (right) using the operator $\hat{H}_{\widetilde{r},n,1}$ and its first derivative $\hat{H}_{\widetilde{r},n,1}^{\prime}$, respectively. The parameter $n$ varies from $n=100$ to $n=1000$ in steps of $50$. This is compared against the standard constrained mock-Chebyshev least squares approximation using the same number of data points.}
 \label{f4fig} 
\end{figure}

Now, we reconstruct the function $f_1$ and its first derivative $f_1^{\prime}$ by using the approximation operator $\hat{H}_{\widetilde{r},n,1}$ and its first derivative $\hat{H}_{\widetilde{r},n,1}^{\prime}$  for $n=25,50,75,100$, respectively. The results are shown in Figures~\ref{fe1fig}-\ref{fe2fig}. It is possible to see that the approximation improves as $n$ increases. More precisely, for $n=25$, the approximation produced by the operators $\hat{H}_{\widetilde{r},n,1}$ and $\hat{H}_{\widetilde{r},n,1}^{\prime}$ present oscillations at the boundary of the domain. These oscillations are then smoothed out and eventually disappear completely for $n=100$.
\begin{figure}[ht]
\centering
\includegraphics[width=.25\textwidth]{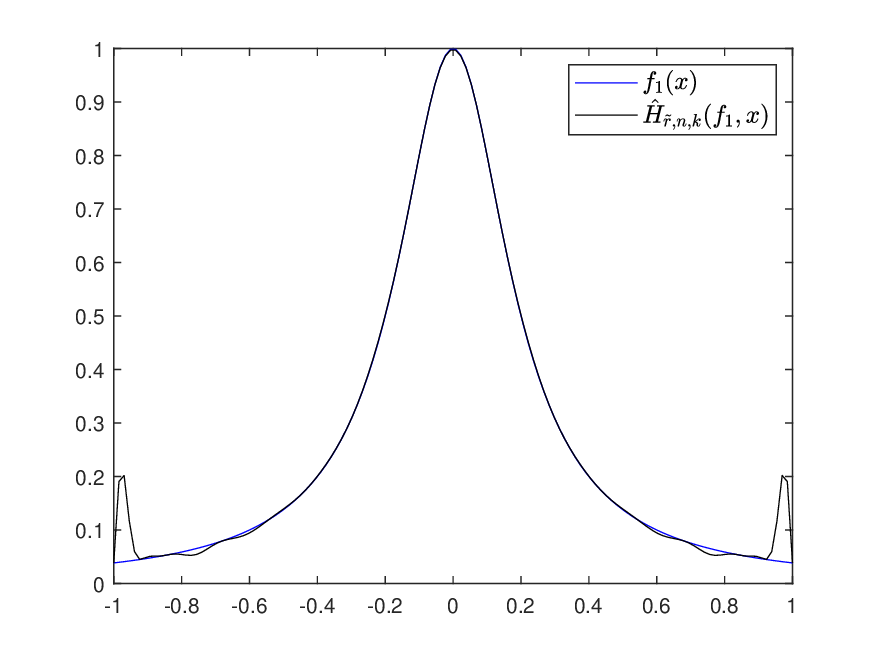}\hfil
\includegraphics[width=.25\textwidth]{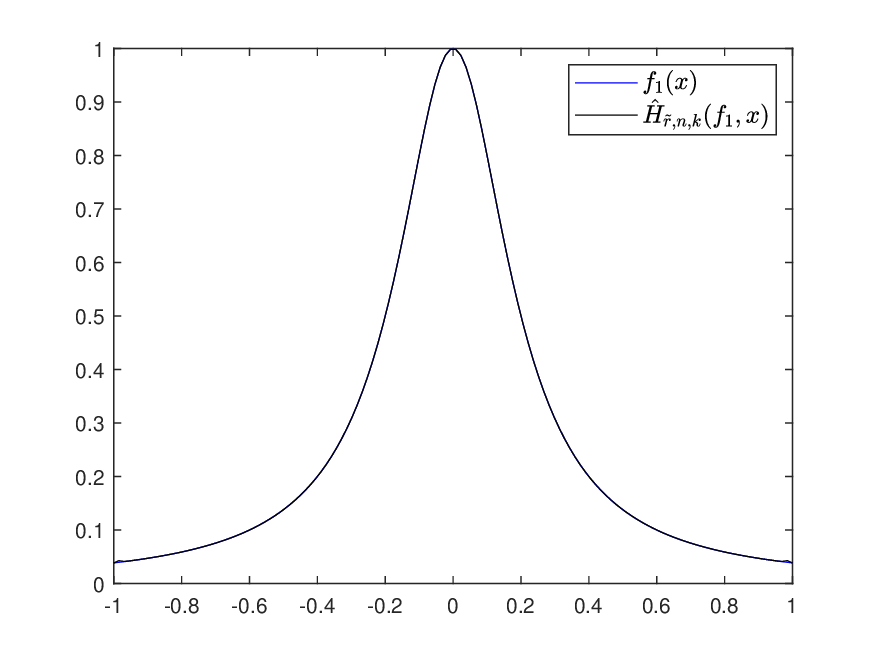}\hfil
\includegraphics[width=.25\textwidth]{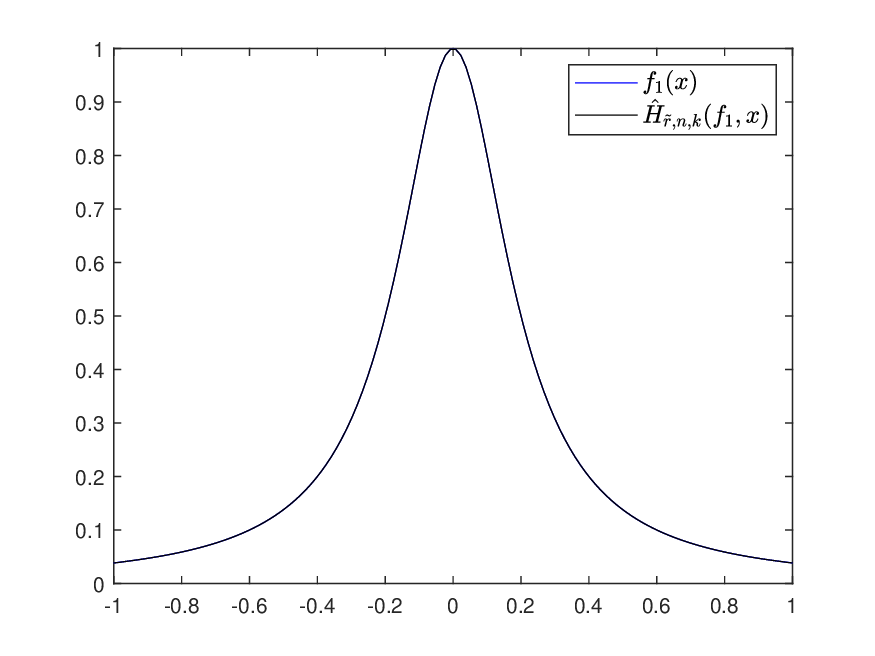}\hfil
\includegraphics[width=.25\textwidth]{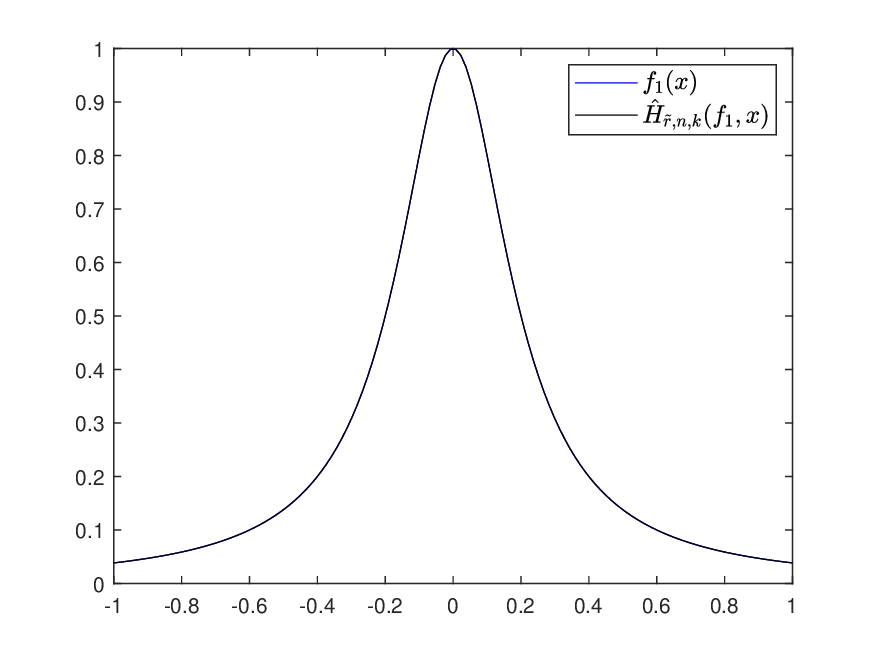}\hfil
    \caption{Reconstruction of the function $f_1(x)$ by using the approximation operator $\hat{H}_{\widetilde{r},n,1}$ for $n=25,50,75,100$ (from left to right).}
 \label{fe1fig} 
\end{figure}
\begin{figure}[ht]
\centering
\includegraphics[width=.25\textwidth]{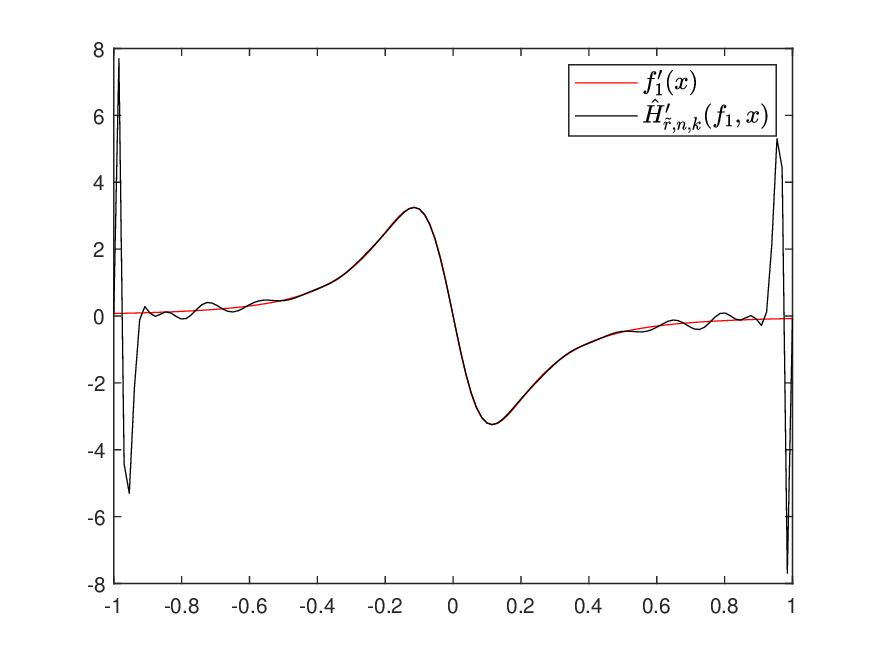}\hfil
\includegraphics[width=.25\textwidth]{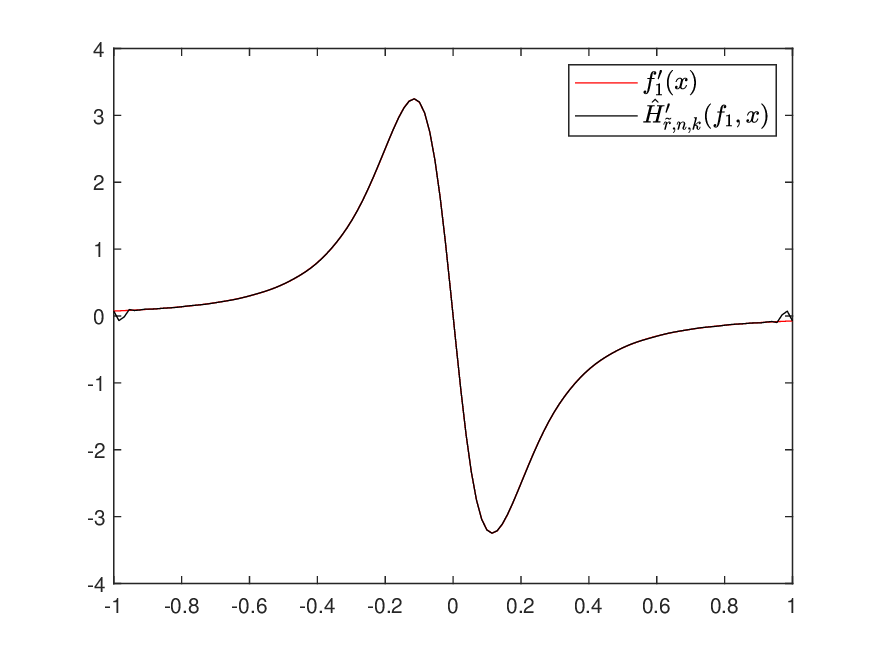}\hfil
\includegraphics[width=.25\textwidth]{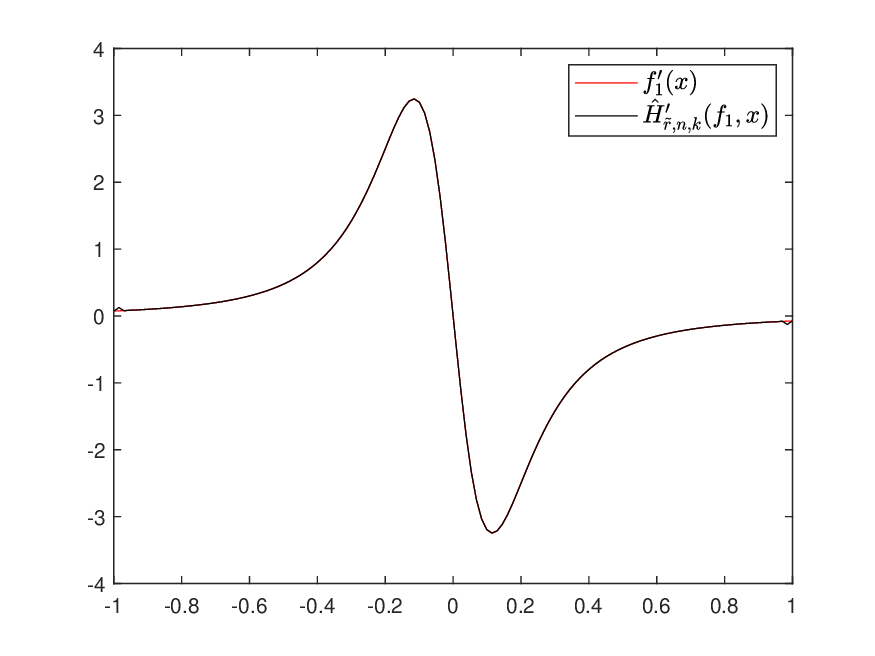}\hfil
\includegraphics[width=.25\textwidth]{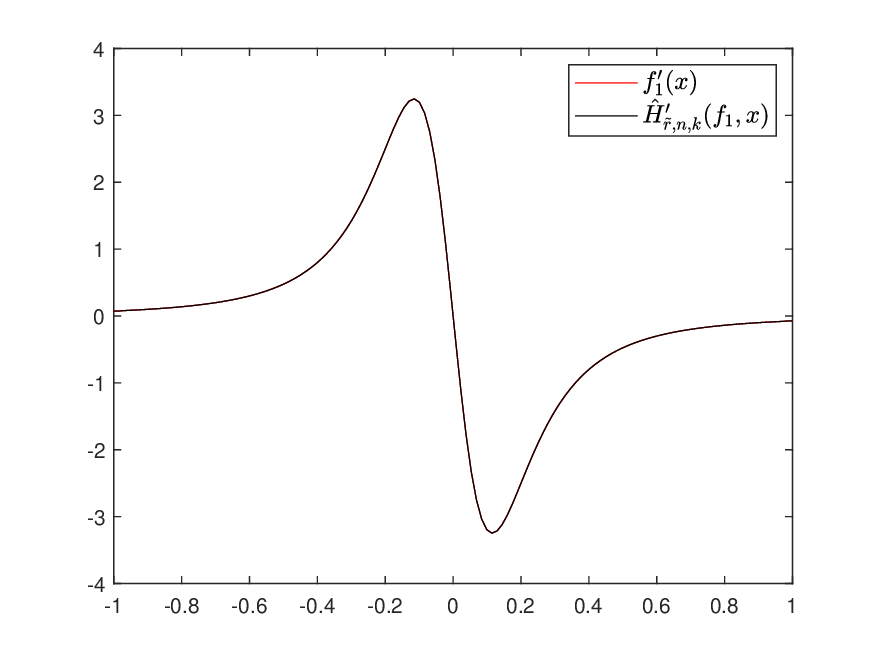}\hfil
    \caption{Reconstruction of the function $f^{\prime}_1(x)$ by using the approximation operator $\hat{H}_{\widetilde{r},n,1}^{\prime}$ for $n=25,50,75,100$ (from left to right).}
 \label{fe2fig} 
\end{figure}

\section{Declarations}
\begin{itemize}
\item Corresponding author: Federico Nudo, email address: federico.nudo@unipd.it
    \item Funding: This research was supported by the GNCS-INdAM 2024 project \lq\lq Metodi kernel e polinomiali per l'approssimazione e l'integrazione: teoria e software applicativo\rq\rq. Additionally, the project was funded by the European Union – NextGenerationEU under the National Recovery and Resilience Plan (NRRP), Mission 4 Component 2 Investment 1.1 - Call PRIN 2022 No. 104 of February 2, 2022, of the Italian Ministry of University and Research; Project 2022FHCNY3 (subject area: PE - Physical Sciences and Engineering) \enquote{Computational mEthods for Medical Imaging (CEMI)}. The work of F. Marcellán has also been supported by the research project PID2021-122154NB-I00 \emph{Ortogonalidad y Aproximación con Aplicaciones en Machine Learning y Teoría de la Probabilidad}, funded by MICIU/AEI/10.13039/501100011033 and by \lq\lq ERDF A Way of Making Europe\rq\rq.
    \item Authors' Contributions: All authors contributed equally to this article, and therefore the order of authorship is alphabetical.
    \item Conflicts of Interest: The authors declare no conflicts of interest.
    \item Ethics Approval: Not applicable.
    \item Data Availability: No data were used in this study.
\end{itemize}

\section*{Acknowledgments}
 This research has been achieved as part of RITA \textquotedblleft Research
 ITalian network on Approximation'' and as part of the UMI group \enquote{Teoria dell'Approssimazione
 e Applicazioni}. The research was supported by GNCS-INdAM 2024 project \lq\lq Metodi kernel e polinomiali per l'approssimazione e l'integrazione: teoria e software applicativo\rq\rq. Project funded by the EuropeanUnion – NextGenerationEU under the National Recovery and Resilience Plan (NRRP), Mission 4 Component 2 Investment 1.1 - Call PRIN 2022 No. 104 of February 2, 2022 of Italian Ministry of University and Research; Project 2022FHCNY3 (subject area: PE - Physical Sciences and Engineering) \enquote{Computational mEthods for Medical Imaging (CEMI)}. The work of F. Marcell\'an has been supported by the research project PID2021- 122154NB-I00] \emph{Ortogonalidad y Aproximación con Aplicaciones en Machine Learning y Teoría de la Probabilidad} funded  by MICIU/AEI/10.13039/501100011033 and by \lq\lq ERDF A Way of making Europe\rq\rq.

 \end{document}